\documentclass[graybox]{svmult}

\usepackage{amsmath,amsbsy,amssymb,latexsym}

\usepackage{mathptmx}
\usepackage{helvet}
\usepackage{courier}
\usepackage{type1cm}

\usepackage{makeidx}
\usepackage{multicol}

\usepackage[bottom]{footmisc}

\usepackage{cite}
\usepackage{url}

\makeindex

\begin{document}

\title*{Fractional Euler-Lagrange differential equations
via Caputo derivatives\thanks{This is a preprint of a paper
whose final and definite form will appear as Chapter~9 of the
book \emph{Fractional Dynamics and Control}, D.~Baleanu et al. (eds.),
Springer New York, 2012, DOI:10.1007/978-1-4614-0457-6\_9, in press.}}

\author{Ricardo Almeida \and Agnieszka B. Malinowska \and Delfim F. M. Torres}

\institute{Ricardo Almeida
\at Center for Research and Development in Mathematics and Applications\\
Department of Mathematics, University of Aveiro, 3810-193 Aveiro, Portugal\\
\email{ricardo.almeida@ua.pt}
\and
Agnieszka B. Malinowska
\at Faculty of Computer Science,
Bia{\l}ystok University of Technology,
15-351 Bia\l ystok, Poland\\
\email{abmalinowska@ua.pt}
\and
Delfim F. M. Torres
\at Center for Research and Development in Mathematics and Applications\\
Department of Mathematics, University of Aveiro, 3810-193 Aveiro, Portugal\\
\email{delfim@ua.pt}}

\maketitle


\abstract{We review some recent
results of the fractional variational calculus.
Necessary optimality conditions of Euler--Lagrange
type for functionals with a Lagrangian
containing left and right Caputo derivatives are given.
Several problems are considered: with fixed or free
boundary conditions, and in presence of
integral constraints that also depend on Caputo derivatives.\\[0.3cm]
{\bf MSC 2010\/}: 26A33; 34K37; 49K05; 49K21.
}


\section{Introduction}

Fractional calculus plays an important role in many different areas,
and has proven to be a truly multidisciplinary subject \cite{Kilbas,Podlubny}.
It is a mathematical field as old as the calculus itself.
In a letter dated 30th September 1695, Leibniz posed
the following question to L'Hopital: ``Can the meaning of derivative
be generalized to derivatives of non-integer order?''
Since then, several mathematicians had investigated Leibniz's challenge,
prominent among them were Liouville, Riemann, Weyl, and Letnikov. There
are many applications of fractional calculus,
\textrm{e.g.}, in viscoelasticity, electrochemistry, diffusion processes,
control theory, heat conduction, electricity, mechanics, chaos and fractals,
and signals and systems \cite{MR2605606,Magin:et:all}.

Several methods to solve fractional differential equations are available,
using Laplace and Fourier transforms, truncated Taylor series,
and numerical approximations. In \cite{Almeida4} a new direct method
to find exact solutions of fractional variational problems is proposed,
based on a simple but powerful idea introduced by Leitmann,
that does not involve solving (fractional) differential equations \cite{Tor:Leit}.
By an appropriate coordinate transformation, we rewrite the initial problem
to an equivalent simpler one; knowing the solution for the new equivalent
problem, and since there exists an one-to-one correspondence between
the minimizers (or maximizers) of the new problem with
the ones of the original, we determine the desired solution.
For a modern account on Leitmann's direct method
see \cite{MyID:183,MyID:187}.

The calculus of variations is a field of mathematics that deals
with extremizing functionals \cite{vanBrunt}.
The variational functionals are often formed as definite integrals
involving unknown functions and their derivatives. The fundamental problem
consists to find functions $y(x)$, $x \in [a,b]$,
that extremize a given functional when subject to
boundary conditions $y(a) = y_a$ and $y(b) = y_b$.
Since this can be a hard task, one wishes to study necessary
and sufficient optimality conditions. The simplest example is the following one:
what is the shape of the curve $y(x)$, $x \in [a,b]$, joining two fixed points
$y_a$ and $y_b$, that has the minimum possible length?
The answer is obviously the straight line joining $y_a$ and $y_b$.
One can obtain it solving the corresponding Euler--Lagrange
necessary optimality condition. If the boundary condition $y(b) = y_b$
is not fixed, \textrm{i.e.}, if we are only interested in the minimum length,
the answer is the horizontal straight line $y(x) = y_a$, $x \in [a,b]$ (free endpoint problem).
In this case we need to complement the Euler--Lagrange equation with an appropriate
natural boundary condition. For a general account on Euler--Lagrange equations and natural boundary
conditions, we refer the reader to \cite{MyID:141,MyID:169} and references therein.
Another important family of variational problems is the isoperimetric one \cite{MyID:136}.
The classical isoperimetric problem consists to find a continuously
differentiable function $y=y(x)$, $x \in [a,b]$,
satisfying given boundary conditions $y(a)=y_a$ and $y(b)=y_b$,
which minimizes (or maximizes) a functional
$$
I(y)=\int_a^b L(x,y(x), y'(x))\, dx
$$
subject to the constraint
$$
\int_a^b g(x,y(x), y'(x))\, dx = l.
$$
The most famous isoperimetric problem can be posed
as follows. Amongst all closed curves with a given length,
which one encloses the largest area? The answer, as we know,
is the circle. The general method to solve such problems
involves an Euler--Lagrange equation obtained via
the concept of Lagrange multiplier
(see, \textrm{e.g.}, \cite{MyID:131}).

The fractional calculus of variations is a recent field,
initiated in 1997, where classical variational problems
are considered but in presence of some fractional
derivative or fractional integral \cite{Riewe:1997}.
In the past few years an increasing of interest has been put on finding necessary
conditions of optimality for variational problems with
Lagrangians involving fractional derivatives \cite{Agrawal,Ata:et:al,Baleanu1,%
Baleanu2,El-Nabulsi:Torres,Frederico:Torres,MyID:089,MyID:149,MyID:163,MyID:181},
fractional derivatives and fractional integrals \cite{MyID:182,Almeida1,MyID:085},
classical and fractional derivatives \cite{MyID:207},
as well as fractional difference operators \cite{MyID:152,MyID:179}.
A good introduction to the subject is given in the monograph \cite{Klimek}.
Here we consider unconstrained and constrained
fractional variational problems via Caputo operators.


\section{Preliminaries and notations}
\label{sec:Prel}

There exist several definitions of fractional derivatives
and fractional integrals, \textrm{e.g.}, Riemann--Liouville, Caputo,
Riesz, Riesz--Caputo, Weyl, Grunwald--Letnikov, Hadamard, and Chen.
Here we review only some basic features of Caputo's fractional derivative.
For proofs and more on the subject, we refer the reader to \cite{Kilbas,Podlubny}.

Let $f:[a,b]\rightarrow\mathbb{R}$ be an integrable function,
$\alpha>0$, and $\Gamma$ be the Euler gamma function.
The left and right Riemann--Liouville fractional
integral operators of order $\alpha$ are defined by\footnote{Along
the work we use round brackets for the arguments of functions,
and square brackets for the arguments of operators. By definition,
an operator receives a function and returns another function.}
$$
{_aI_x^\alpha}[f] := x \mapsto \frac{1}{\Gamma(\alpha)}\int_a^x (x-t)^{\alpha-1}f(t)dt
$$
and
$$
{_xI_b^\alpha}[f] := x \mapsto \frac{1}{\Gamma(\alpha)}\int_x^b(t-x)^{\alpha-1} f(t)dt,
$$
respectively. The left and right Riemann--Liouville fractional derivative
operators of order $\alpha$ are, respectively, defined by
$$
{_aD_x^\alpha} :=\frac{d^n}{dx^n} \circ {_aI_x^{n-\alpha}}
$$
and
$$
{_xD_b^\alpha}:=(-1)^n \frac{d^n}{dx^n} \circ {_xI_b^{n-\alpha}},
$$
where $n=[\alpha]+1$. Interchanging the composition
of operators in the definition of Riemann--Liouville fractional derivatives,
we obtain the left and right Caputo fractional derivatives
of order $\alpha$:
$$
{_a^CD_x^\alpha} := {_aI_x^{n-\alpha}} \circ \frac{d^n}{dx^n}
$$
and
$$
{_x^CD_b^\alpha}:={_xI_b^{n-\alpha}} \circ (-1)^n \frac{d^n}{dx^n}.
$$

\begin{theorem}
Assume that $f$ is of class $C^n$ on $[a,b]$.
Then its left and right Caputo derivatives are
continuous on the closed interval $[a,b]$.
\end{theorem}

One of the most important results for the proof
of necessary optimality conditions,
is the integration by parts formula.
For Caputo derivatives the following relations hold.

\begin{theorem}
\label{thm:frac:IP:C}
Let $\alpha>0$, and $f, g:[a,b]\to\mathbb{R}$ be $C^n$ functions.
Then,
\begin{multline*}
\int_{a}^{b}g(x)\cdot {_a^C D_x^\alpha}[f](x)dx
=\int_a^b f(x)\cdot {_x D_b^\alpha}[g](x)dx\\
+\sum_{j=0}^{n-1}\left[{_xD_b^{\alpha+j-n}}[g](x)
\cdot {_xD_b^{n-1-j}}[f](x)\right]_a^b
\end{multline*}
and
\begin{multline*}
\int_{a}^{b}g(x)\cdot {_x^C D_b^\alpha}[f](x)dx
=\int_a^b f(x)\cdot {_a D_x^\alpha}[g](x)dx\\
+ \sum_{j=0}^{n-1} \left[(-1)^{n+j}{_aD_x^{\alpha+j-n}}[g](x)
\cdot {_aD_x^{n-1-j}}[f](x)\right]_a^b,
\end{multline*}
where ${_aD_x^{k}}={_aI_x^{-k}}$ and
${_xD_b^{k}}={_xI_b^{-k}}$ whenever $k<0$.
\end{theorem}

In the particular case when $0<\alpha<1$,
we get from Theorem~\ref{thm:frac:IP:C} that
\begin{equation*}
\int_{a}^{b}g(x)\cdot {_a^C D_x^\alpha}[f](x)dx
=\int_a^b f(x)\cdot {_x D_b^\alpha}[g](x)dx
+\left[{_xI_b^{1-\alpha}}[g](x) \cdot f(x)\right]_a^b
\end{equation*}
and
\begin{equation*}
\int_{a}^{b}g(x)\cdot {_x^C D_b^\alpha}[f](x)dx
=\int_a^b f(x)\cdot {_a D_x^\alpha}[g](x)dx
-\left[{_aI_x^{1-\alpha}}[g](x) \cdot f(x)\right]_a^b.
\end{equation*}
In addition, if $f$ is such that $f(a)=f(b)=0$, then
$$
\int_{a}^{b}  g(x)\cdot{_a^C D_x^\alpha}[f](x)dx
=\int_a^b f(x)\cdot {_x D_b^\alpha}[g](x)dx
$$
and
$$
\int_{a}^{b}  g(x)\cdot{_x^C D_b^\alpha}[f](x)dx
=\int_a^b f(x)\cdot {_a D_x^\alpha}[g](x)dx.
$$

Along the work, we denote by $\partial_iL$, $i=1,\ldots,m$
($m\in\mathbb{N}$), the partial derivative of function
$L:\mathbb{R}^m\rightarrow \mathbb{R}$
with respect to its $i$th argument.
For convenience of notation, we introduce the operator
$^C_\alpha[\cdot]_{\beta}$ defined by
$$
^C_\alpha[y]_{\beta}
:= x \mapsto \left(x,y(x),\, {_a^C D_x^\alpha}[y](x),\,{_x^C D_b^\beta}[y](x)\right),
$$
where $\alpha,\beta\in(0,1)$.


\section{Euler--Lagrange equations}
\label{sec:EulerLag}

The fundamental problem of the fractional calculus of variations
is addressed in the following way: find functions $y \in \mathcal{E}$,
$$
\mathcal{E}:=\left\{ y\in C^1([a,b]) \, | \, y(a)=y_a \mbox{ and } y(b)=y_b \right\},
$$
that maximize or minimize the functional
\begin{equation}
\label{funct1}
J(y)=\int_a^b \left(L\circ {^C_\alpha}[y]_{\beta}\right)(x)dx.
\end{equation}
As usual, the Lagrange function $L$ is assumed to be of class $C^1$
on all its arguments. We also assume that
$\partial_3 L \circ {^C_\alpha}[y]_{\beta}$ has continuous
right Riemann--Liouville fractional derivative of order $\alpha$
and $\partial_4 L \circ {^C_\alpha}[y]_{\beta}$
has continuous left Riemann--Liouville
fractional derivative of order $\beta$ for $y \in \mathcal{E}$.

In \cite{Agrawal} a necessary condition of optimality for such functionals is proved.
We remark that although functional \eqref{funct1} contains only Caputo fractional derivatives,
the fractional Euler--Lagrange equation also contains
Riemann--Liouville fractional derivatives.

\begin{theorem}[Euler--Lagrange equation for (\ref{funct1})]
\label{thm:3}
If $y$ is a minimizer or a maximizer of $J$ on $\mathcal{E}$, then $y$
is a solution of the fractional differential equation
\begin{equation}
\label{ELequation}
\left(\partial_2 L \circ {^C_\alpha}[y]_{\beta}\right)(x)
+{_xD_b^\alpha}\left[\partial_3 L \circ {^C_\alpha}[y]_{\beta}\right](x)
+{_aD_x^\beta}\left[\partial_4 L \circ {^C_\alpha}[y]_{\beta}\right](x)=0
\end{equation}
for all $x\in[a,b]$.
\end{theorem}

\begin{proof}
Given $|\epsilon| \ll 1$, consider $h \in V$ where
$$
V := \left\{ h\in C^1([a,b]) \, | \, h(a)=0 \mbox{ and } h(b)=0 \right\},
$$
and a variation of function $y$ of type $y+\epsilon h$.
Define the real valued function $j(\epsilon)$ by
\begin{equation*}
j(\epsilon)=J(y+\epsilon h)
=\int_a^b \left(L \circ {^C_\alpha}[y+\epsilon h]_{\beta}\right)(x) dx.
\end{equation*}
Since $\epsilon=0$ is a minimizer or a maximizer of $j$, we have $j'(0)=0$. Thus,
\begin{multline*}
\int_a^b \Bigl[
\left(\partial_2 L \circ {^C_\alpha}[y]_{\beta}\right)(x) \cdot h(x)
+ \left(\partial_3 L \circ {^C_\alpha}[y]_{\beta}\right)(x)
\cdot {_a^CD_x^\alpha}[h](x)\\
+ \left(\partial_4 L \circ {^C_\alpha}[y]_{\beta}\right)(x)
\cdot {_x^CD_b^\beta}[h](x) \Bigr]dx=0.
\end{multline*}
We obtain equality \eqref{ELequation}
integrating by parts and applying the
classical fundamental lemma of the calculus of variations \cite{vanBrunt}.
\end{proof}

We remark that when $\alpha \rightarrow 1$, then \eqref{funct1}
is reduced to a classical functional
$$
J(y)=\int_a^b f(x,y(x),y'(x))\,dx,
$$
and the fractional Euler--Lagrange equation \eqref{ELequation}
gives the standard one:
$$
\partial_2 f(x,y(x),y'(x))-\frac{d}{dx} \partial_3 f(x,y(x),y'(x))=0.
$$
Solutions to equation \eqref{ELequation}
are said to be \emph{extremals} of \eqref{funct1}.


\section{The isoperimetric problem}
\label{sec:Iso}

The fractional isoperimetric problem is stated in the following way:
find the minimizers or maximizers of functional $J$ as in \eqref{funct1},
over all functions $y \in \mathcal{E}$
satisfying the fractional integral constraint
$$
I(y)=\int_a^b \left(g \circ {^C_\alpha}[y]_{\beta}\right)(x) dx = l.
$$
Similarly as $L$, $g$ is assumed to be of class $C^1$ with respect
to all its arguments, function $\partial_3 g \circ {^C_\alpha}[y]_{\beta}$
is assumed to have continuous
right Riemann--Liouville fractional derivative of order $\alpha$
and $\partial_4 g \circ {^C_\alpha}[y]_{\beta}$
continuous left Riemann--Liouville
fractional derivative of order $\beta$ for $y \in \mathcal{E}$.
A necessary optimality condition for the fractional
isoperimetric problem is given in \cite{Almeida3}.

\begin{theorem}
\label{thm:iso:n}
Let $y$ be a minimizer or maximizer of $J$ on $\mathcal{E}$,
when restricted to the set of functions $z \in \mathcal{E}$ such that
$I(z)=l$. In addition, assume that $y$ is not an extremal of $I$.
Then, there exists a constant $\lambda$ such that $y$ is a solution of
\begin{equation}
\label{eq:EL:iso}
\left(\partial_2 F \circ {^C_\alpha}[y]_{\beta}\right)(x)
+{_xD_b^\alpha}[\partial_3 F \circ {^C_\alpha}[y]_{\beta}](x)
+ {_aD_x^\beta}[\partial_4 F \circ {^C_\alpha}[y]_{\beta}](x)=0
\end{equation}
for all $x \in [a,b]$, where $F=L+\lambda g$.
\end{theorem}

\begin{proof}
Given $h_1,h_2 \in V$, $|\epsilon_1|\ll1$ and $|\epsilon_2|\ll1$, consider
$$
j(\epsilon_1,\epsilon_2)
=\int_a^b \left(L \circ {^C_\alpha}[y+\epsilon_1h_1+\epsilon_2h_2]_{\beta}\right)(x)dx
$$
and
$$
i(\epsilon_1,\epsilon_2)
=\int_a^b \left(g \circ {^C_\alpha}[y+\epsilon_1h_1+\epsilon_2h_2]_{\beta}\right)(x)dx-l.
$$
Since $y$ is not an extremal for $I$, there exists a function $h_2$ such that
$$\left.\frac{\partial i}{\partial \epsilon_2} \right|_{(0,0)}\neq 0,$$
and by the implicit function theorem, there exists
a $C^1$ function $\epsilon_2(\cdot)$,
defined in some neighborhood of zero, such that
$$
i(\epsilon_1,\epsilon_2(\epsilon_1))=0.
$$
Applying the Lagrange multiplier rule
(see, \textrm{e.g.}, \cite[Theorem~4.1.1]{vanBrunt})
there exists a constant $\lambda$ such that
$$
\nabla(j(0,0)+\lambda i(0,0))=\textbf{0}.
$$
Differentiating $j$ and $i$ at $(0,0)$,
and integrating by parts, we prove the theorem.
\end{proof}

\begin{example}
Let $\overline y(x)={E_{\alpha}}(x^\alpha)$, $x \in [0,1]$,
where ${E_\alpha}$ is the Mittag--Leffler function.
Then ${^C_0 D_x^\alpha}[\overline{y}]=\overline{y}$.
Consider the following fractional variational problem:
\begin{equation*}
\begin{gathered}
J(y)=\int_0^1 \left({^C_0 D_x^\alpha}[y](x)\right)^2 \, dx
\longrightarrow \textrm{extr},\\
I(y)=\int_0^1 \overline{y}(x) \, {^C_0 D_x^\alpha}[y](x) \, dx  = l,\\
y(0)=1\, , \quad y(1)= y_1,
\end{gathered}
\end{equation*}
with $l := \displaystyle \int_0^1 (\overline{y}(x))^2 dx$
and $y_1 := {E_\alpha}(1)$. In this case function
$F$ of Theorem~\ref{thm:iso:n} is
$$
F(x,y,v,w) =v^2+\lambda \overline{y}(x) \, v
$$
and the fractional Euler--Lagrange equation \eqref{eq:EL:iso} is
$$
{_xD_1^\alpha}[2\,{^C_0D_x^\alpha}[y]+\lambda \overline{y}](x)=0.
$$
A solution to this problem is $\lambda=-2$
and $y(x)=\overline{y}(x)$, $x \in [0,1]$.
\end{example}

The case when $y$ is an extremal of $I$
is also included in the results of \cite{Almeida3}.

\begin{theorem}
If $y$ is a minimizer or a maximizer of $J$ on $\mathcal{E}$,
subject to the isoperimetric constraint $I(y)=l$,
then there exist two constants $\lambda_0$
and $\lambda$, not both zero, such that
$$
\left(\partial_2 K \circ {^C_\alpha}[y]_{\beta}\right)(x)
+{_xD_b^\alpha}\left[\partial_3 K \circ {^C_\alpha}[y]_{\beta}\right](x)
+{_aD_x^\beta}\left[\partial_4 K \circ {^C_\alpha}[y]_{\beta}\right](x)=0
$$
for all $x \in [a,b]$, where $K=\lambda_0 L+\lambda g$.
\end{theorem}

\begin{proof}
The same as the proof of Theorem~\ref{thm:iso:n},
but now using the abnormal Lagrange multiplier rule
(see, \textrm{e.g.}, \cite[Theorem~4.1.3]{vanBrunt}).
\end{proof}


\section{Transversality conditions}
\label{trans}

We now give the \emph{natural boundary conditions}
(also known as \emph{transversality conditions})
for problems with the terminal point
of integration free as well as $y_b$.

Let
$$
\mathcal{F}:=\left\{ (y,x)\in C^1([a,b])\times [a,b] \, | \, y(a)=y_a \right\}.
$$
The type of functional we consider now is
\begin{equation}
\label{eq:JT}
J(y,T)=\int_a^T \left(L \circ {^C_\alpha}[y]\right)(x)\,dx,
\end{equation}
where the operator $^C_\alpha[\cdot]$ is defined by
$$
^C_\alpha[y]:= x \mapsto \left(x,y(x),\, {_a^C D_x^\alpha}[y](x)\right).
$$
These problems are investigated in \cite{Agrawal}
and more general cases in \cite{Almeida2}.

\begin{theorem}
Suppose that $(y,T) \in \mathcal{F}$
minimizes or maximizes $J$ defined
by \eqref{eq:JT} on $\mathcal{F}$. Then
\begin{equation}
\label{eq:EL:T}
\left(\partial_2 L \circ {^C_\alpha}[y]\right)(x)
+{_x D_T^\alpha}\left[\partial_3 L \circ {^C_\alpha}[y]\right](x)=0
\end{equation}
for all $ x \in [a,T]$. Moreover, the following transversality conditions hold:
$$
\left(L \circ {^C_\alpha}[y]\right)(T) =0\, ,
\quad
{_x I_T^{1-\alpha}}\left[\partial_3 L \circ {^C_\alpha}[y] \right](T)=0.
$$
\end{theorem}

\begin{proof}
The result is obtained by considering variations $y+\epsilon h$
of function $y$ and variations $T+\epsilon \triangle T$
of $T$ as well, and then applying the Fermat theorem,
integration by parts, Leibniz's rule, and using the arbitrariness of
$h$ and $\triangle T$.
\end{proof}

Transversality conditions for several other situations can
be easily obtained. Some important examples are:

\begin{itemize}
\item If $T$ is fixed but $y(T)$ is free,
then besides the Euler--Lagrange equation \eqref{eq:EL:T}
one obtains the transversality condition
$$
{_x I_T^{1-\alpha}}\left[\partial_3 L \circ {^C_\alpha}[y]\right](T)=0.
$$

\item If $y(T)$ is given but $T$ is free, then the transversality condition is
$$
\left(L \circ {^C_\alpha}[y]\right)(T)
-y'(T) \cdot {_x I_T^{1-\alpha}}\left[\partial_3 L
\circ {^C_\alpha}[y]\right](T)=0.
$$

\item If $y(T)$ is not given but is restricted
to take values on a certain given curve $\psi$,
\textrm{i.e.}, $y(T)=\psi(T)$, then
$$
\left(\psi'(T) - y'(T) \right) \cdot
{_x I_T^{1-\alpha}}\left[\partial_3 L \circ {^C_\alpha}[y]\right](T)
+ \left(L \circ {^C_\alpha}[y]\right)(T)=0.
$$
\end{itemize}


\begin{acknowledgement}
Work supported by {\it FEDER} funds through
{\it COMPETE} --- Operational Programme Factors of Competitiveness
(``Programa Operacional Factores de Competitividade'')
and by Portuguese funds through the
{\it Center for Research and Development
in Mathematics and Applications} (University of Aveiro)
and the Portuguese Foundation for Science and Technology
(``FCT --- Funda\c{c}\~{a}o para a Ci\^{e}ncia e a Tecnologia''),
within project PEst-C/MAT/UI4106/2011
with COMPETE number FCOMP-01-0124-FEDER-022690.
Agnieszka Malinowska is also supported by Bia{\l}ystok
University of Technology grant S/WI/2/2011.
\end{acknowledgement}



\end{document}